\documentclass[reqno]{amsart}
\usepackage{amsmath}
\usepackage[dvips]{graphicx}
\usepackage{amsfonts}
\usepackage{amssymb}
\usepackage{latexsym}
\graphicspath{{Img/}}

\newtheorem{theorem}{Theorem}

\theoremstyle{plain}

\newtheorem{corollary}{Corollary}

\newtheorem{example}{Example}

\newtheorem{problem}{Problem}
\newtheorem{proposition}{Proposition}
\newtheorem{remark}{Remark}

\numberwithin{equation}{section}

\begin{document}
\title[The twist angle of weighted Steiner minimal trees]{The twist angle of weighted Steiner minimal trees in the three dimensional Euclidean Space}
\author{Anastasios N. Zachos}

\address{Greek Ministry of Education, Greece}
\email{azachos@gmail.com} 
\begin{abstract}
We find the equations that allow us to compute the position of the two interior nodes (weighted Fermat-Torricelli points) w.r. to the  weighted Steiner problem for four points determining a tetrahedron in $\mathbb{R}^{3}.$ Furthermore, by applying the solution w.r. to the weighted Steiner problem for a boundary tetrahedron, we calculate the twist angle between the two weighted Steiner planes formed by one edge and the line defined by the two weighted Fermat-Torricelli points and a non-neighbouring edge and the line defined by the two weighted Fermat-Torricelli points.
\end{abstract}\maketitle

\section{Introduction}

We state the weighted Steiner problem for $n$ points in $\mathbb{R}^{3}.$
The Steiner problem for equal weights $\mathbb{R}^{2}$ has been introduced in (\cite[pp.360]{Cour/Rob:51}).

\begin{problem}
Given n points $A_{1},\cdots A_{n},$ in $\mathbb{R}^{3},$ such that a positive real number (weights) corresponds to each point $A_{i},$  to find a connected system of
straight line segments of shortest weighted total length such that any two
of the given points can be joined by a polygon consisting of
segments of the system.
\end{problem}
A weighted Steiner tree is a solution of the weighted Steiner problem in $\mathbb{R}^{3}.$
For a classical study about the geometry of the weighted Steiner problem and further generalizations on manifolds,  we refer to \cite{IvanovTuzhilin:01}, \cite{IvanovTuzhilin:94}, \cite{Ci} and for equal weights, we refer to \cite{GilbertPollak:68}.

The solution of the Steiner tree problem determines some tree topologies, which have been
characterized by the following theorem:

\begin{theorem}{\cite[pp.~328]{BolMa/So:99},\cite{GilbertPollak:68}}
Any solution of the unweighted Steiner problem is a tree (a
Steiner tree) with at most $n-2$ Fermat-Torricelli points, where
each Fermat-Torricelli point has degree three and the angles
formed between any two edges incident with a Fermat-Torricelli
point are equal ($120^{\circ}$). The $n-2$ Fermat-Torricelli
points are vertices of the polygonal tree which do not belong to
$\{A_{1},\cdots,A_{n}\}.$
\end{theorem}

In 2002, Rubinstein, Thomas and Weng (\cite{RubinsteinThomasWeng:02}) succeeded in calculating the two interior Fermat-Torricelli points (Steiner nodes) for a boundary tetrahedron in $\mathbb{R}^{3},$ by using two Cartesian Coordinate systems w.r to the common perpendicular of two non-neighbouring edges of the tetrahedron, in order to determine the position of line which passes through the two nodes (Simpson line). Furthermore, they provide the necessary and sufficient conditions for the existence of the two non-degenarate weighted Fermat-Torricelli points of w.r. to a boundary tetrahedron.
In 2006, Mese and Yamada (\cite{MeseYamada:06})gave a beautiful connection between the Steiner problem and the Singular Plateau problem via the energy functional. The weighted Steiner problem may also be called as a one dimensional Plateau problem and the corresponding solutions may be called branching solutions (\cite{IvanovTuzhilin:01}).
In 2007, J.M. Smith, Jang and Kim apply the connection of the minimum energy configuration with unweighted Steiner trees and study the geometric folding problem of proteins by measuring twist angles and considering various energy configurations. The protein folding problem is to take a linear string of aminoacids and examine the geometric properties and how it folds in $\mathbb{R}^{3}.$ The twist angle is the angle formed by two planes composed by a node (Fermat-Torricelli point) and an edge (Steiner planes). According to our knowledge, there is not any known formula to compute the twist angle of the weighted (or unweighted) Steiner problem for boundary tetrahedra in $\mathbb{R}^{3}.$

By replacing zero weights in $n-3$ vertices of Problem~1, we derive the classical weighted Fermat-Torricelli problem for a triangle in $\mathbb{R}^{2},$ which has been introduced and first solved by Gueron and Tessler in \cite{Gue/Tes:02}.

In this paper, we study the geometry of the weighted Steiner problem for four points determining a tetrahedron in $\mathbb{R}^{3}.$

In Section~2, we provide  the necessary and sufficient conditions for the existence of two weighted Fermat-Torricelli points (nodes) at the interior of the tetrahedron in $\mathbb{R}^{3}$ (Theorem~2).

Furthermore, we find the equations that allow us to compute the position of the weighted Simpson line w.r to the weighted Steiner problem for four points determining tetrahedra in $\mathbb{R}^{3}$ (Theorem~4) and we find the equations that allow us to compute the position of the unweighted Fermat-Torricelli point w.r to the unweighted Fermat-Torricelli problem for a boundary tetrahedron in $\mathbb{R}^{3}$ (Theorem~5).
It is worth mentioning that the equations taken from Theorem~5 which deal with the computation of the position of the unweighted Fermat-Torricelli problem for a boundary tetrahedron are more complicated than the equations taken from Proposition~1 to determine the position for the unweighted Simpson line w.r to the unweighted Steiner problem for the same boundary tetrahedron.

In Section~3, we derive a formula to compute the twist angle of the weighted (or unweighted) Steiner planes for
a boundary tetrahedron $A_{1}A_{2}A_{3}A_{3}$ (Theorem~6). The twist angle is the dihedral angle which is formed by an edge $A_{1}A_{2}$ and the position of the weighted Simpson line and the edge $A_{4}A_{4}$ and the position of the weighted Simpson line. The weighted Simpson line is the line defined by the two weighted Fermat-Torricelli points.


\section{The weighted Steiner problem for four points determining tetrahedra in the three-dimensional Euclidean Space.}

\begin{figure}
\centering
\includegraphics[scale=0.90]{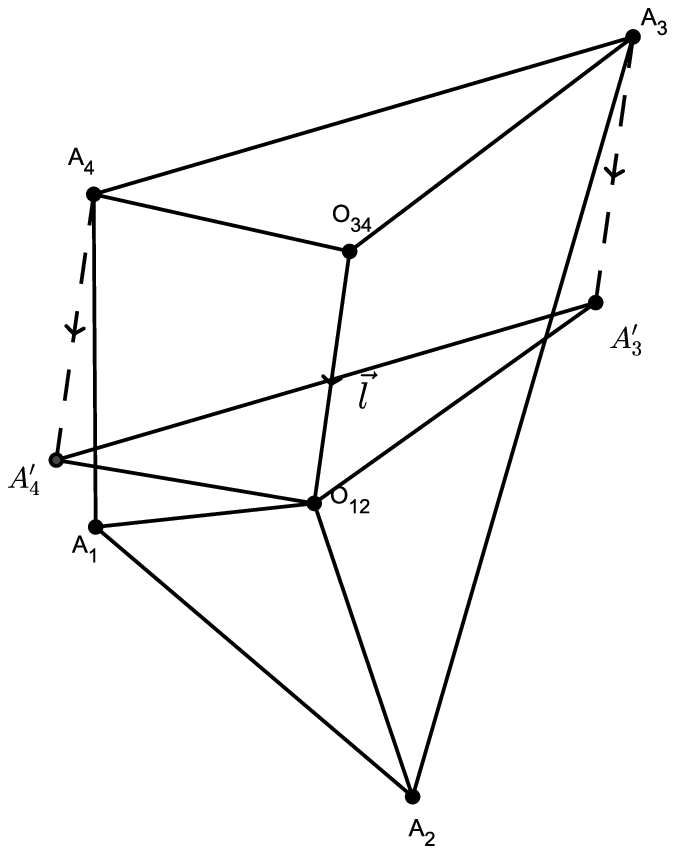}
\caption{} \label{figg1}
\end{figure}

\begin{figure}
\centering
\includegraphics[scale=0.90]{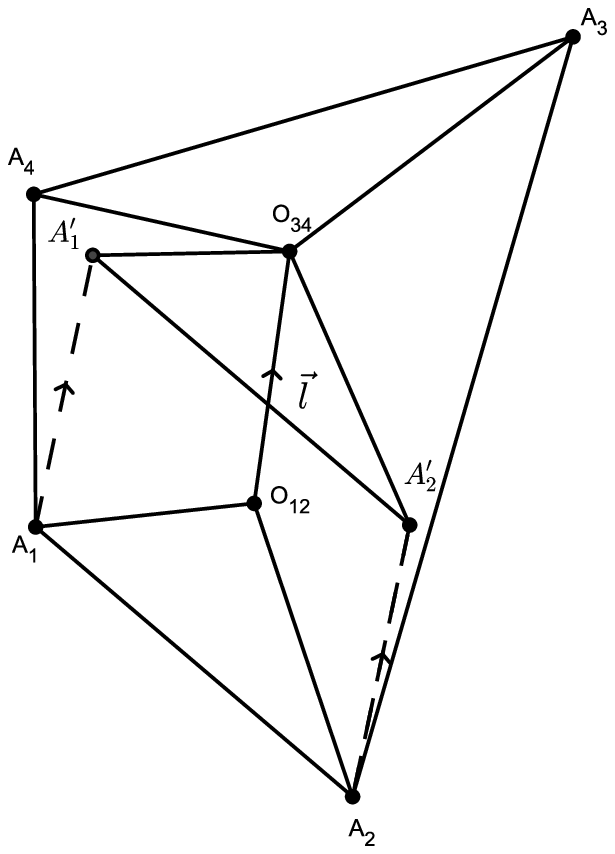}
\caption{} \label{figg2}
\end{figure}

We denote by $A_{1}A_{2}A_{3}A_{4}$ a tetrahedron in $\mathbb{R}^{3},$ with
$A_{i}(x_{i},y_{i},z_{i})$ ($i=1,2,3,4$), by $B_{i}$ a positive real number(weight) which corresponds to the vertex $A_{i},$ $O_{12},$ $O_{34}$ two interior points (nodes) of $A_{1}A_{2}A_{3}A_{4}$ in $\mathbb{R}^{3},$
by $B_{12}$ the weight which corresponds to $O_{12},$ $B_{34}$ the weight which corresponds to $O_{34},$ by $H$ the length of the common perpendicular (Euclidean distance) between the two lines defined by $A_{1}A_{2},$ $A_{4}A_{3},$
by $A_{i}A_{j}$ the Euclidean distance from $A_{i}$ to $A_{j},$ by $O_{12}O_{34}$ the Euclidean distance from $O_{12}$ to $O_{34},$ by $A_{i}O_{12}$ the Euclidean distance from $A_{i}$ to $O_{12}$ and by $A_{j}O_{34}$ the Euclidean distance from $A_{j}$ to $O_{34},$ for $i,j=1,2,3,4.$

Furthermore, we set:
$\vec{l}\equiv \overrightarrow{O_{12}O_{34}},$
$\vec{a}_{1}\equiv
\overrightarrow{A_{1}O_{12}},$
$\vec{a}_{2}\equiv\overrightarrow{A_{2}O_{12}},$
$\vec{a}_{3}\equiv\overrightarrow{A_{3}O_{34}},$
$\vec{a}_{4}\equiv\overrightarrow{A_{4}O_{34}}$ and
$\vec{a}_{ij}\equiv \overrightarrow{A_{i}A_{j}},$ for
$i,j=1,2,3,4,$ $i\ne j\ne k,$ $\alpha_{12}\equiv \angle A_{1}O_{12}A_{2},$
$\alpha_{34}\equiv \angle A_{3}O_{34}A_{4},$ $\alpha_{1}\equiv
\angle A_{2}O_{12}O_{34},$ $\alpha_{2}\equiv \angle
A_{1}O_{12}O_{34},$  $\alpha_{3}\equiv \angle A_{4}O_{34}O_{12},$
$\alpha_{4}\equiv \angle A_{3}O_{34}O_{12}$ and
$\varphi\equiv \arccos(\frac{\vec{a}_{12}\cdot \vec{a}_{43}}{a_{12}a_{43}}).$

We assume that:
$A_{1}A_{4}+A_{2}A_{3}>A_{1}A_{2}+A_{3}A_{4}$(Fig.~\ref{figg1}).

The weighted Steiner problem for $A_{1}A_{2}A_{3}A_{4}$ in $\mathbb{R}^{3}$
states that:

\begin{problem}\label{Steinertetrahedron}
Find $O_{12}(x_{0},y_{0},z_{0})$ and
$O_{34}(x_{0^{\prime}},y_{0^{\prime}},z_{0^{\prime}})$ with given
weights $B_{12}$ in $O_{12}$ and $B_{34}$ in $O_{34},$ such that
\begin{equation}\label{equat1}
f(O_{12},O_{34})=B_{1}A_{1}O_{12}+ B_{2}A_{2}O_{12}+B_{3}A_{3}O_{34}+B_{4}A_{4}O_{34}+\frac{B_{12}+B_{34}}{2}O_{12}O_{34}\to min.
\end{equation}
\end{problem}



We continue by mentioning the necessary and sufficient conditions for the existence of the two non-degenerate weighted Fermat-Torricelli points $O_{12}$ and $O_{34},$ (weighted Steiner points).

\begin{theorem}\label{conditionswst}
The following inequalities provide the necessary and sufficient conditions  for the existence of the two non-degenerate weighted Fermat-Torricelli points $O_{12}$ and $O_{34}:$
\begin{equation}\label{ineq1}
\frac{\sqrt{y(C)^2+z(C)^2}}{x_{1}-x(C)}>\tan(\arccos(\frac{(\frac{B_{12}+B_{34}}{2})^2-B_{1}^2-B_{2}^2)}{2 B_{1}B_{2}}))
\end{equation}
if $x(C)<x_{1},$
\begin{equation}\label{ineq2}
\frac{\sqrt{y(C)^2+z(C)^2}}{x(C)-x_{1}}>\tan(\arccos(\frac{(\frac{B_{12}+B_{34}}{2})^2-B_{1}^2-B_{2}^2)}{2 B_{1}B_{2}}))
\end{equation}
if $x_{1}<x(C),$
\begin{equation}\label{ineq3}
(\sqrt{y(C)^2+z(C)^2}+R_{12}\sin\beta_{12})^{2}+(\frac{x_{1}+x_{2}}{2}-x(C))^2>R_{12}^{2}
\end{equation}

where

$R_{12}\equiv \frac{A_{1}A_{2}}{(B_{1}+B_{2}+\frac{B_{12}+B_{34}}{2})(B_{1}+B_{2}-\frac{B_{12}+B_{34}}{2})(B_{2}+\frac{B_{12}+B_{34}}{2}-B_{1})(B_{1}+\frac{B_{12}+B_{34}}{2}-B_{2})}$

$\beta_{12}=\arccos(\frac{A_{1}A_{2}}{2R_{12}}),$

and

\begin{equation}\label{ineq4}
\frac{\sqrt{y(C)^2+z(C)^2}}{x_{4}-x(C)}>\tan(\arccos(\frac{(\frac{B_{12}+B_{34}}{2})^2-B_{3}^2-B_{4}^2)}{2 B_{3}B_{4}}))
\end{equation}
if $x(C)<x_{4},$
\begin{equation}\label{ineq5}
\frac{\sqrt{y(C)^2+z(C)^2}}{x(C)-x_{4}}>\tan(\arccos(\frac{(\frac{B_{12}+B_{34}}{2})^2-B_{3}^2-B_{4}^2)}{2 B_{3}B_{4}}))
\end{equation}
if $x_(1)<x(C),$
\begin{equation}\label{ineq6}
(\sqrt{y(C)^2+z(C)^2}+R_{34}\sin\beta_{34})^{2}+(\frac{x_{4}+x_{3}}{2}-x(C))^2>R_{34}^{2},
\end{equation}

where

$R_{34}\equiv \frac{A_{4}A_{3}}{(B_{3}+B_{4}+\frac{B_{12}+B_{34}}{2})(B_{3}+B_{4}-\frac{B_{12}+B_{34}}{2})(B_{3}+\frac{B_{12}+B_{34}}{2}-B_{4})(B_{4}+\frac{B_{12}+B_{34}}{2}-B_{3})}$

$\beta_{34}=\arccos(\frac{A_{4}A_{3}}{2R_{34}}).$

\end{theorem}

\begin{proof}
We extend the methodology that has been used in \cite{RubinsteinThomasWeng:02}, to find the necessary and sufficient conditions to locate the two non-degenerate weighted Fermat-Torricelli points $O_{12}$ and $O_{34}$ at the interior of $A_{1}A_{2}A_{3}A_{4}.$

Without loss of generality, we may assume that $A_{1},$ $A_{2}$ lie on the x-Axis and satisfy $x_{1}<x_{2}.$
Let $C=(x(C),y(C),z(C)) \in \mathbb{R}^{3}.$ The point $C$ lies outside the right circular half cone $C_{A_{1}}$ whose vertex is $A_{1}$ and whose axis is $A_{1}x$ is consistent with the condition $\angle A_{2}A_{1}C < \angle A_{2}O_{12}O_{34}=\arccos(\frac{B_{1}^2-(\frac{B_{12}+B_{34}}{2})^2-B_{2}^2}{2 B_{2}\frac{B_{12}+B_{34}}{2}}).$

We set $C^{\prime}=(x(C),0,0)$

From $\triangle A_{1}C^{\prime}C,$ we obtain (\ref{ineq1}).
If $x(C)>x_{1},$ by working similarly, the condition to locate $C$ outside the right circular half cone with vertex $A_{2},$ gives (\ref{ineq2}).

Let $C_{A_{1}A_{2}}$ be the surface obtained by revolving $A_{1}A_{2}$ about $A_{1}A_{2}.$ The condition
$\angle A_{2}A_{1}C < \angle A_{2}O_{12}O_{34}$ is equivalent to $C$ lying outside the surface $C_{A_{1}A_{2}}.$
Hence, we get:
\[(\sqrt{y(C)^2+z(C)^2}+R_{12}\sin(\arccos(\frac{A_{1}A_{2}}{2R_{12}})))^{2}+(\frac{x_{1}+x_{2}}{2}-x_{1})^2>R_{12}^2,\]
which gives (\ref{ineq3}).
The circumradius $R_{12}$ is the radius of the inscribed circle w.r. to the triangle $\triangle A_{1}A_{12}A_{2}$
having lengths $A_{1}A_{2}=\lambda \frac{B_{12}+B_{34}}{2},$ $A_{1}A_{12}=\lambda B_{2}$ and $A_{2}A_{12}=\lambda B_{1},$ where $\lambda=\frac{A_{1}A_{2}}{\frac{B_{12}+B_{34}}{2}}.$ By following a similar process, we derive (\ref{ineq4}), ((\ref{ineq5})) and (\ref{ineq6}).
\end{proof}

By replacing $B_{1}=B_{2}=B_{3}=B_{4}=B_{12}=B_{34}=1,$ in Theorem~2, we derive the necessary and sufficient conditions for the existence of two non-degenerate unweighted Fermat-Torricelli points of the Steiner problem
for boundary tetrahedra. These conditions have been introduced in \cite{RubinsteinThomasWeng:02}.

\begin{corollary}\cite{RubinsteinThomasWeng:02}
The following inequalities provide the necessary and sufficient conditions  for the existence of the two non-degenerate unweighted Fermat-Torricelli points $O_{12}$ and $O_{34}:$
\begin{equation}\label{ineq1ust}
\frac{\sqrt{y(C)^2+z(C)^2}}{x_{1}-x(C)}>\sqrt{3}
\end{equation}
if $x(C)<x_{1},$
\begin{equation}\label{ineq2ust}
\frac{\sqrt{y(C)^2+z(C)^2}}{x_{C}-x_{1}}>\sqrt{3}
\end{equation}
if $x_{1}<x(C),$
\begin{equation}\label{ineq3ust}
(\sqrt{y(C)^2+z(C)^2}+\frac{r_{12}}{3})^{2}+(\frac{x_{1}+x_{2}}{2}-x(C))^2>(\frac{2r_{12}}{3})^{2}
\end{equation}

where

$r_{12}=\frac{\sqrt{3}}{2} A_{1}A_{2}$

and

\begin{equation}\label{ineq4ust}
\frac{\sqrt{y(C)^2+z(C)^2}}{x_{4}-x(C)}>\sqrt{3}
\end{equation}
if $x(C)<x_{4},$
\begin{equation}\label{ineq5ust}
\frac{\sqrt{y(C)^2+z(C)^2}}{x_{C}-x_{4}}>\sqrt{3}
\end{equation}
if $x_{4}<x(C),$
\begin{equation}\label{ineq6ust}
(\sqrt{y(C)^2+z(C)^2}+(\frac{r_{34}}{3})^{2}+(\frac{x_{4}+x_{3}}{2}-x(C))^2>(\frac{2r_{34}}{3})^{2},
\end{equation}

where

$r_{34}=\frac{\sqrt{3}}{2} A_{4}A_{3}.$

\end{corollary}

\begin{theorem}
The solution of the weighted Steiner problem is a weighted Steiner tree in $\mathbb{R}^{3}$
whose nodes $O_{12}$ and $O_{34}$ (weighted Fermat-Torricelli points) are seen by the angles:
\begin{eqnarray}\label{FTangles}
 \cos\alpha_{12}& = & \frac{B_{ST}^2-B_{1}^2-B_{2}^2}{2B_{1}B_{2}}\nonumber \\
\cos\alpha_{1} & = & \frac{B_{1}^2-B_{2}^2-B_{ST}^2}{2B_{2}B_{ST}}\nonumber \\
\cos\alpha_{34} & = &\frac{B_{ST}^2-B_{3}^2-B_{4}^2}{2B_{3}B_{4}}\nonumber \\
\cos\alpha_{4} & =& \frac{B_{4}^2-B_{3}^2-B_{ST}^2}{2B_{3}B_{ST}}.
\end{eqnarray}

\end{theorem}

\begin{proof}

By differentiating (\ref{equat1}) with respect to $x_{0}, y_{0},
z_{0}, x_{0^{\prime}}, y_{0^{\prime}}, z_{0^{\prime}}$ we get:

\begin{equation}\label{equat2}
B_{1}\frac{\vec{a}_{1}}{\|\vec{a}_{1}\|}+B_{2}\frac{\vec{a}_{2}}{\|\vec{a}_{2}\|}-\frac{B_{12}+B_{34}}{2}\frac{\vec{l}}{\|\vec{l}\|}=0
\end{equation}

and

\begin{equation}\label{equat3}
B_{3}\frac{\vec{a}_{3}}{\|\vec{a}_{3}\|}+B_{4}\frac{\vec{a}_{4}}{\|\vec{a}_{4}\|}+\frac{B_{12}+B_{34}}{2}\frac{\vec{l}}{\|\vec{l}\|}=0.
\end{equation}

By adding (\ref{equat2}) and (\ref{equat3}) we have:

\begin{equation}\label{equat4}
\sum_{i=1}^{4}B_{i}\frac{\vec{a}_{i}}{\|\vec{a}_{i}\|}=0
\end{equation}

which is the translation of $\triangle A_{4}O_{34}A_{3}\to
\triangle A_{4^{\prime}}O_{12}A_{3^{\prime}}$ with respect to
$\vec{l}$ or $\triangle A_{1}O_{12}A_{2}\to \triangle
A_{1^{\prime}}O_{34}A_{2^{\prime}}$ with respect to $\vec{l}$ (see
Fig.~\ref{figg1}).

The solution of (\ref{equat4}) with respect to the tetrahedron
$A_{1}A_{2}A_{3^{\prime}}A_{4^{\prime}}$ or
$A_{1^{\prime}}A_{2^{\prime}}A_{3}A_{4}$ (see \cite{Zach/Zou:09})
explains two properties of (\ref{equat1}):

(1) the invariance of $O_{12}$ or $O_{34}$

(2) the translation $O_{12}\to O_{34}$ or $O_{34}\to O_{12}.$

The problem is that $\vec{l}$ is unknown and (\ref{equat4}) cannot
help for the solution of (\ref{equat1}).

The equations (\ref{equat2}) and (\ref{equat3}), which give
$O_{12}$ and $O_{34}$ the weighted Fermat-Torricelli points of
$\triangle A_{1}A_{2}O_{34}$ and $\triangle A_{3}A_{4}O_{12},$
respectively, explain that $A_{1}A_{2}O_{12}O_{34}$ and
$A_{3}A_{4}O_{12}O_{34}$ are two planes and the line defined by
$O_{12}O_{34}$ intersect $A_{1}A_{2}$ and $A_{3}A_{4}$ at the
points $T_{12}$ and $T_{34},$ respectively (Fig.~\ref{figg2}). We
set $B_{ST}\equiv \frac{B_{12}+B_{34}}{2}$

\begin{figure}
\centering
\includegraphics[scale=0.90]{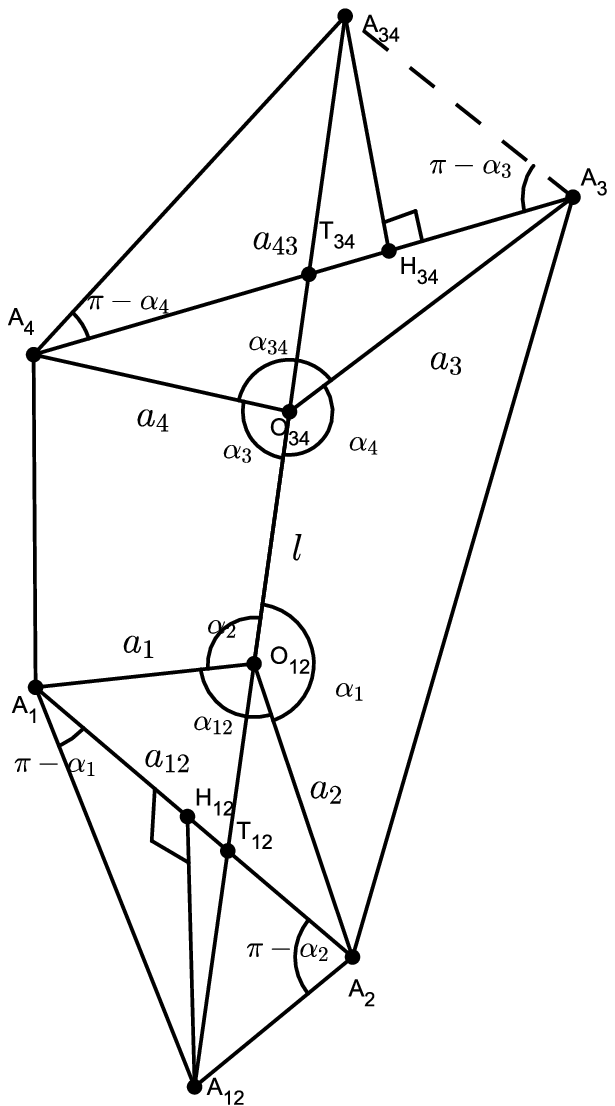}
\caption{} \label{figg2}
\end{figure}

The solution of (\ref{equat2}) is (Fig.~\ref{figg2}, see
\cite{Zach/Zou:08}):
\begin{equation}\label{equat5}
\frac{B_{1}}{\sin\alpha_{1}}=\frac{B_{2}}{\sin\alpha_{2}}=\frac{B_{ST}}{\sin\alpha_{12}}.
\end{equation}

The solution of (\ref{equat3}) is (Fig.~\ref{figg2}, see
\cite{Zach/Zou:08}):
\begin{equation}\label{equat6}
\frac{B_{3}}{\sin\alpha_{3}}=\frac{B_{4}}{\sin\alpha_{4}}=\frac{B_{ST}}{\sin\alpha_{34}}.
\end{equation}

From (\ref{equat5}) and (\ref{equat6}), we obtain (\ref{FTangles}).











\end{proof}

The Euclidean distance of the common perpendicular $H$ is given by:
\begin{equation}\label{equat8}
H =\frac{ \left|
\begin{array}{ccccc}
x_{4}-x_{1}      & y_{4}-y_{1}      & z_{4}-z_{1}   \\
x_{2}-x_{1}      & y_{2}-y_{1}      & z_{2}-z_{1}          \\
x_{3}-x_{4}      & y_{3}-y_{4}      & z_{3}-z_{4}          \\

\end{array} \right|}{a_{12}a_{34}\sin\varphi},
\end{equation}

We denote by $A_{4}^{\prime \prime}$ the intersection point of  the line defined by the $A_{4}A_{3}$ and the line defined by the common perpendicular of $A_{1}A_{2}$ and $A_{4}A_{3}$ and by $A_{1}^{\prime \prime}$ the intersection point of  the line defined by $A_{1}A_{2}$ and the line defined by the common perpendicular of $A_{1}A_{2}$ and $A_{4}A_{3},$ by $T_{12}$ the intersection point of the line defined by $O_{12}O_{34}$ and the line defined by $A_{1}A_{2}$ and by $T_{34}$ the intersection point of the line defined by $O_{12}O_{34}$ and the line defined by $A_{4}A_{3},$ $M_{12}$ the midpoint of $A_{1}A_{2}$ and $M_{34}$ the midpoint of $A_{4}A_{3}.$

We assume that $A_{1}^{\prime\prime}\notin [A_{1},A_{2}]$ and $A_{4}^{\prime \prime}\notin [A_{4},A_{3}].$
and $A_{1}^{\prime\prime}$ lies on the half line w.r to $A_{1}$ and $A_{4}^{\prime\prime}$ lies on the half line w.r. to $A_{4}.$

We denote by $T_{12}^{\prime}$ the orthogonal projection of $T_{12}$ to the parallel line of the line defined by $A_{3}A_{4}$ and by $A_{4}^{\prime}$ the point which lies on the half line w.r. to $A_{4},$ such that: $H=A_{4}^{\prime}T_{12}^{\prime},$ $\phi=\angle T_{12}A_{1}^{\prime\prime}T_{12}^{\prime}$ and $A_{1}^{\prime\prime}T_{12}^{\prime}A_{4}^{\prime}A_{4}^{\prime\prime}$ is a parallelogram, by $A_{12}$ the vertex of $\triangle A_{1}A_{12}A_{2},$ such that: $\angle A_{1}A_{12}A_{2}=\pi - \alpha_{12},$ $\angle A_{12}A_{1}A_{2}=\pi-\alpha_{1}$ and $\angle A_{1}A_{2}A_{12}=\pi-\alpha_{2},$ by $A_{34}$ the vertex of $\triangle A_{4}A_{34}A_{3},$ such that: $\angle A_{4}A_{34}A_{3}=\pi - \alpha_{34},$ $\angle A_{34}A_{4}A_{3}=\pi-\alpha_{4}$ and $\angle A_{4}A_{3}A_{34}=\pi-\alpha_{3},$ by $H_{12}$ the trace of the height of $\triangle A_{1}A_{12}A_{2}$ w.r to the base $A_{1}A_{2}$ and by $A_{34}$ the vertex of $\triangle A_{4}A_{34}A_{3},$ such that: $\angle A_{4}A_{34}A_{3}=\pi - \alpha_{34},$ $\angle A_{34}A_{4}A_{3}=\pi-\alpha_{4}$ and $\angle A_{4}A_{3}A_{34}=\pi-\alpha_{3}$ and by $H_{34}$ the trace of the height of $\triangle A_{4}A_{34}A_{3}$ w.r to the base $A_{4}A_{3}.$

Furthermore, we set $H\equiv A_{4}^{\prime \prime}A_{1}^{\prime \prime},$ $t_{34}\equiv A_{4}^{\prime\prime} T_{34}$
$t_{12}\equiv A_{1}^{\prime \prime} T_{12}$ $k_{1}\equiv A_{1}^{\prime\prime} A_{1}$ and $k_{2}\equiv A_{4}^{\prime\prime} A_{4},$ $m_{12}\equiv A_{1}^{\prime\prime}M_{12}$ and $m_{34}\equiv A_{4}^{\prime\prime}M_{34},$ $h_{12}^{\prime}\equiv A_{1}^{\prime\prime}H_{12}$ and $h_{34}^{\prime}\equiv A_{4}^{\prime\prime}H_{34}.$

\begin{theorem}\label{impthm1st}
The following system of the two equations w.r. to $t_{34}$ and $t_{12}$ allow us to compute the position of
the weighted Simpson line $O_{12}O_{34}$ of the weighted full Steiner tree for $A_{1}A_{2}A_{3}A_{4}:$
\begin{equation}\label{firsteq1wst}
\frac{t_{34}-t_{12}\cos\phi}{\sqrt{H^2+t_{12}\sin^{2}\phi}}=\frac{h_{34}^{\prime}-t_{34}}{r_{34}}
\end{equation}

and

\begin{equation}\label{secondeq2wst}
\frac{t_{12}-t_{34}\cos\phi}{\sqrt{H^2+t_{34}\sin^{2}\phi}}=\frac{h_{12}^{\prime}-t_{12}}{r_{12}}
\end{equation}

\end{theorem}
\begin{proof}


From the similarity of triangles $\triangle A_{4}^{\prime}T_{12}T_{34},$ $\triangle A_{34}H_{34}T_{34}$

($\triangle A_{4}^{\prime}T_{12}T_{34} \sim \triangle A_{34}H_{34}T_{34}$) we obtain (\ref{firsteq1wst}).

By applying the sine law in $\triangle A_{4}A_{34}A_{3},$ we get:

\begin{equation}\label{sinwst1}
A_{3}A_{34}=a_{34}\frac{\sin(\pi-\alpha_{4})}{\sin(\pi-\alpha_{34})}
\end{equation}

and by replacing (\ref{equat5}) in (\ref{sinwst1}), we get:

\begin{equation}\label{sinwst1n}
A_{3}A_{34}=a_{34}\frac{B_{4}}{B_{0}}
\end{equation}

By applying the sine law in $\triangle H_{34}A_{34}A_{3}$ and taking into account (\ref{sinwst1n}) we get:

\begin{equation}\label{radius34}
r_{34}=\frac{B_{4}}{B_{ST}}a_{34}\sin(\arccos(\frac{B_{3}^2-B_{ST}^2-B_{4}^2}{2 B_{ST}B_{4}}))
\end{equation}

From $\triangle A_{4}H_{34}A_{34},$ we derive:

\begin{equation}\label{sinwst3n}
A_{4}H_{34}=r_{34}\cot(\pi-\alpha_{4}).
\end{equation}

Thus, we get:

\begin{equation}\label{calch34prime}
h_{34}^{\prime}=A_{4}^{\prime\prime}A_{4}+r_{34}\cot(\pi-\alpha_{4}).
\end{equation}

By replacing (\ref{calch34prime}) and (\ref{radius34}) in (\ref{firsteq1wst}), we derive the first equation
which depends on $t_{12},$ $t_{34}$ and the given weights $B_{3},$ $B_{4}$ and $B_{ST}.$

By working similarly, and by exchanging the indices $4\to 1,$ and $3\to 2,$ we obtain (\ref{secondeq2wst})
which depends on $t_{12},$ $t_{34}$ and the given weights $B_{1},$ $B_{2}$ and $B_{ST}.$

\end{proof}

\begin{remark}
By solving (\ref{firsteq1wst}) w.r. to $t_{34}$ and (\ref{secondeq2wst})  w.r. to $t_{12},$
we derive that $t_{34}=f_{34}(t_{12})$ and $t_{12}=f_{12}(t_{34}).$
By setting $t_{12}(0)=m_{12}$ we obtain an iteration process $t_{34}(0)=f_{34}(m_{12}),$ $t_{12}(1)=f_{12}(t_{34}(0))=f_{12}(f_{34}(m_{12})),$ $t_{34}(1)=f_{34}(t_{12}(1))=f_{34}(f_{12}(f_{34}(m_{12})))\cdots.$

We assume that $t_{12}=t_{12}(n)=f_{12}f_{34}\cdots (n times)f_{12}f_{34}(m_{12})$
and $t_{34}=t_{34}(n)=f_{34}(f_{12}f_{34}\cdots (ntimes)(f_{34}(m_{12}))).$
\end{remark}

\begin{proposition}\label{prop1stn}
\begin{equation}\label{firsteq1st}
\frac{t_{34}-t_{12}\cos\phi}{\sqrt{H^2+t_{12}\sin^{2}\phi}}=\frac{m_{34}-t_{34}}{a_{34}\frac{\sqrt{3}}{2}}
\end{equation}
and

\begin{equation}\label{secondeq2st}
\frac{t_{12}-t_{34}\cos\phi}{\sqrt{H^2+t_{34}\sin^{2}\phi}}=\frac{m_{12}-t_{12}}{a_{12}\frac{\sqrt{3}}{2}}
\end{equation}
\end{proposition}

\begin{proof}
By replacing $B_{1}=B_{2}=B_{3}=B_{4}=B_{ST}=1$ in (\ref{firsteq1wst})
and (\ref{secondeq2wst}), we derive (\ref{firsteq1st}) and (\ref{secondeq2st}), respectively.
\end{proof}

\begin{remark}
We note that Proposition~1 has been established in \cite{RubinsteinThomasWeng:02} by using two coordinate systems w.r to $A_{4}^{\prime\prime}$ and $A_{1}^{\prime\prime}$ and
by solving (\ref{firsteq1st}) w.r. to $t_{34}$ and (\ref{secondeq2st})  w.r. to $t_{12},$
they derive that $t_{34}=g_{34}(t_{12}),$ $t_{12}=g_{12}(t_{34})$
and by setting $t_{12}(0)=A_{1}^{\prime\prime}A_{1}$ or $t_{34}=A_{4}^{\prime\prime}A_{4},$ they obtain an iteration sequence $t_{34}(0)=g_{34}(A_{1}^{\prime\prime}A_{1}),$ $t_{12}(1)=g_{12}(t_{34}(0))=g_{12}(g_{34}(A_{1}^{\prime\prime}A_{1})),$ $t_{34}(1)=g_{34}(t_{12}(1))=g_{34}(g_{12}(g_{34}(A_{1}^{\prime\prime}A_{1})))\cdots.$
The convergence of this iteration sequence is proved in \cite{RubinsteinThomasWeng:02} by showing that it is monotone decreasing.

\end{remark}

We denote by $F$ the corresponding Fermat-Torricelli point of $A_{1}A_{2}A_{3}A_{4}$ having the property:
\[A_{1}F+A_{2}F+A_{3}F+A_{4}F\to min.\]
The solution of this objective function yields the (unweighted) Fermat-Torricelli tree $\{A_{1}F,A_{2}F,A_{3}F,A_{4}F\}.$

\begin{theorem}
The following system of the three equations w.r. to $t_{34},$ $t_{12}$ and $\angle A_{4}FA_{3}$ allow us to compute the position of the line defined by $T_{12}T_{34}$ of the (unweighted) Fermat-Torricelli tree of $A_{1}A_{2}A_{3}A_{4}:$

\begin{equation}\label{firsteq1ft}
\frac{t_{34}-t_{12}\cos\phi}{\sqrt{H^2+t_{12}\sin^{2}\phi}}=\frac{m_{34}-t_{34}}{\frac{a_{34}}{2}\tan\frac{\angle A_{4}FA_{3}}{2}},
\end{equation}

\begin{equation}\label{secondeq2ft}
\frac{t_{12}-t_{34}\cos\phi}{\sqrt{H^2+t_{34}\sin^{2}\phi}}=\frac{m_{12}-t_{12}}{\frac{a_{12}}{2}\tan\frac{\angle A_{4}FA_{3}}{2}}
\end{equation}

and

\begin{equation}\label{calcft3}
\cot\frac{\angle A_{4}FA_{3}}{2}=\frac{2(H^2+k_{1}(t_{12}-t_{34}\cos\varphi))+k_{2}(t_{34}-t_{12}\cos\varphi)}{(t_{12}-k_{1})\sqrt{H^2+t_{34}^2\sin^{2}\varphi}+(t_{34}-k_{2})\sqrt{H^2+t_{12}^2\sin^{2}\varphi}}
\end{equation}
\end{theorem}

\begin{proof}
The line (Simpson line ) which passes through $F$ intersects $A_{1}A_{2}$ at $T_{12}$ and $A_{4}A_{3}$ at $T_{34}$ and $\angle A_{4}FA_{3}=\angle A_{1}FA_{2},$ such that: $\angle A_{4}A_{3}A_{34}=\angle A_{3}A_{4}A_{34}=\frac{ A_{4}FA_{3}}{2},$ and $\angle A_{1}A_{2}A_{12}=\angle A_{2}A_{1}A_{12}=\frac{ A_{4}FA_{3}}{2}.$ Furthermore, we obtain $H_{34}=M_{34}$ and $H_{12}=M_{12}.$ These equalities yield $h_{34}^{\prime}=m_{34}$ and $h_{12}^{\prime}=m_{12}.$

From the similarity of triangles $\triangle A_{4}^{\prime}T_{12}T_{34},$ $\triangle A_{34}M_{34}T_{34}$

($\triangle A_{4}^{\prime}T_{12}T_{34} \sim \triangle A_{34}H_{34}T_{34}$) we derive (\ref{firsteq1ft}).

The second equation shall be derived by using the following relations:

\begin{equation}\label{sft1n}
T_{12}T_{34}=T_{12}F+FT_{34},
\end{equation}
\begin{equation}\label{sft2n}
T_{12}F=(t_{12}-k_{1})(\sin\varphi_{12}\cot\frac{ A_{4}FA_{3}}{2}+\cos\varphi_{12})
\end{equation}
\begin{equation}\label{sft3n}
T_{34}F=(t_{34}-k_{2})(\sin\varphi_{34}\cot\frac{ A_{4}FA_{3}}{2}+\cos\varphi_{34})
\end{equation}
\begin{equation}\label{sft4n}
\sin\varphi_{12}=\frac{\sqrt{h^2+t_{34}^2\sin^{2}\varphi}}{T_{12}T_{34}}
\end{equation}
\begin{equation}\label{sft5n}
\cos\varphi_{12}=\frac{t_{12}-t_{34}\cos\varphi}{T_{12}T_{34}}
\end{equation}
\begin{equation}\label{sft6n}
\sin\varphi_{34}=\frac{\sqrt{H^2+t_{12}^2\sin^{2}\varphi}}{T_{12}T_{34}}
\end{equation}
\begin{equation}\label{sft7n}
\cos\varphi_{34}=\frac{t_{34}-t_{12}\cos\varphi}{T_{12}T_{34}}
\end{equation}

By applying the cosine law in $\triangle A_{4}^{\prime\prime}T_{12}T_{34},$
we get:

\begin{equation}\label{sft8n}
\cos\varphi_{34}=\frac{t_{34}^2+T_{12}T_{34}^2-(H^2)+t_{12}^2}{2 t_{34}T_{12}T_{34}}.
\end{equation}

From (\ref{sft7n}) and (\ref{sft8n}), we obtain:
\begin{equation}\label{sft9n}
T_{12}T_{34}=\sqrt{H^2+t_{12}^2+t_{34}^2-2t_{12}t_{34}\cos\varphi}
\end{equation}
\end{proof}

By replacing (\ref{sft2n}), (\ref{sft3n}),(\ref{sft4n}),(\ref{sft5n}),(\ref{sft6n}),(\ref{sft7n}),(\ref{sft8n}), (\ref{sft9n})
in (\ref{sft1n}), we derive (\ref{calcft3}).

\begin{example}

Le $A_{1}A_{2}A_{3}A_{4}$ be a tetrahedron in $\mathbb{R}^{3},$ such that:

$A_{1}=(0,0,0),$ $A_{2}=(2,0,0),$ $A_{3}=(-2,0,3),$ $A_{4}=(-1,-1,2),$ $h\approx 2.12,$ $\varphi \approx 125.26^{\circ}.$ By replacing (\ref{firsteq1ft}), (\ref{secondeq2ft}) and (\ref{calcft3}), we derive that:

$t_{12}\approx 1.17,$ $t_{34} \approx 1.42$ which yields $\alpha \approx 43.47^{\circ}.$

\end{example}

\begin{remark}
If the common perpendicular of $A_{1}A_{2}$ and $A_{4}A_{3}$ intersect the two linear segments at $A_{1}^{\prime\prime}\in [A_{1},A_{2}]$ and $A_{4}^{\prime\prime}\in [A_{4},A_{3}]$
or $A_{1}^{\prime\prime}\notin [A_{1},A_{2}]$ and $A_{4}^{\prime\prime}\in [A_{4},A_{3}]$
the following equations allow us to compute the position of the unweighted Fermat-Torricelli point $F$ for $B_{1}=B_{2}=B_{3}=B_{4}:$

\[\frac{\|t_{34}-t_{12}\cos\phi\|}{\sqrt{H^2+t_{12}\sin^{2}\phi}}=\frac{\|m_{34}-t_{34}\|}{\frac{a_{34}}{2}\tan\frac{\angle A_{4}FA_{3}}{2}},\]

\[\frac{\|t_{12}-t_{34}\cos\phi\|}{\sqrt{H^2+t_{34}\sin^{2}\phi}}=\frac{\|m_{12}-t_{12}\|}{\frac{a_{12}}{2}\tan\frac{\angle A_{4}FA_{3}}{2}}\]

\[\cot\frac{\angle A_{4}FA_{3}}{2}=\frac{2(H^2+k_{1}(t_{12}-t_{34}\cos\varphi))+k_{2}(t_{34}-t_{12}\cos\varphi)}{(\|t_{12}-k_{1}\|)\sqrt{H^2+t_{34}^2\sin^{2}\varphi}+(\|t_{34}-k_{2}\|)\sqrt{H^2+t_{12}^2\sin^{2}\varphi}}\]

\end{remark}

\section{The twist angle formed by the planes $A_{1}A_{2}T_{12}T_{34}$ and $A_{4}A_{3}T_{34}T_{12}$}

We denote by $\omega$ the dihedral angle (twist angle) formed by the planes $A_{1}A_{2}T_{12}T_{34}$ and $A_{4}A_{3}T_{34}T_{12},$
by $\vec{u_{12}}$ the unit vector from $A_{1}$ to $A_{2},$ by $\vec{u_{34}}$ the unit vector from $A_{3}$ to $A_{4}$
and by $\vec{u_{12S}}$ the unit vector from $T_{12}$ to $T_{34}.$

We recall that $\varphi_{12}=\angle A_{1}T_{12}T_{34}$ and $\varphi_{34}=\angle A_{4}T_{34}T_{12}.$

\begin{theorem}\label{calctwistangle1}
The twist angle $\omega$ is given by
\begin{equation}\label{calctwistangle2}
\cos\omega=\frac{\cos\varphi-\cos\varphi_{12}\cos\varphi_{34}}{\sin\varphi_{12}\varphi_{34}}.
\end{equation}
\end{theorem}

\begin{proof}
We take into account the following identity w.r to the vectors $\vec{a},$ $\vec{b},$ $\vec{c}$ and $\vec{d}:$
\begin{equation}\label{calctwistangle3}
(\vec{a}\times \vec{b})\cdot (\vec{c}\times \vec{d})=(\vec{a}\cdot \vec{c})(\vec{b}\cdot \vec{d})-(\vec{a}\cdot \vec{d})(\vec{b}\cdot \vec{c}).
\end{equation}

By replacing $\vec{a}=\vec{u_{12}},$ $\vec{b}=\vec{d}=\vec{u_{12S}}$ and $\vec{c}=\vec{u_{34}}$ in (\ref{calctwistangle3}), we obtain:
\begin{equation}\label{calctwistangle4}
\|\vec{u_{12}}\times\vec{u_{12S}}\| \|\vec{u_{34}}\times\vec{u_{12S}}  \| \cos\omega=(\vec{u_{12}}\cdot\vec{u_{34}})\|\vec{u_{12S}}\|^{2}-(\vec{u_{12}}\cdot\vec{u_{12S}})(\vec{u_{34}}\vec{u_{12S}})
\end{equation}
or

\begin{equation}\label{calctwistangle5}
\sin(\pi-\varphi_{12})\sin(\pi-\varphi_{34})\cos\omega=\cos\varphi-\cos(\pi-\varphi_{12})\cos(\pi-\varphi_{34})
\end{equation}
which yields (\ref{calctwistangle2}). By replacing (\ref{sft5n}), (\ref{sft7n})
and (\ref{sft9n})

\[\cos\varphi_{12}=\frac{t_{12}-t_{34}\cos\varphi}{T_{12}T_{34}}\]

\[\cos\varphi_{34}=\frac{t_{34}-t_{12}\cos\varphi}{T_{12}T_{34}},\]

\[T_{12}T_{34}=\sqrt{H^2+t_{12}^2+t_{34}^2-2t_{12}t_{34}\cos\varphi}\]

in (\ref{calctwistangle2}), the twist angle $\omega$ depends on $t_{12}$ and $t_{34}.$

\end{proof}

\begin{corollary}
If $\varphi=0^{\circ},$ then
\begin{equation}\label{calctwistangle6n}
\cos\omega=\frac{1-\cos\varphi_{12}\cos\varphi_{34}}{\sin\varphi_{12}\varphi_{34}}.
\end{equation}
\end{corollary}

\begin{proof}
By replacing $\varphi=0^{\circ}$ in (\ref{calctwistangle2}), we derive (\ref{calctwistangle6n}).
\end{proof}

\begin{corollary}
If $\varphi=90^{\circ},$ then

\begin{equation}\label{calctwistangle7n}
\cos\omega=-\cot\varphi_{12}\cot\varphi_{34}.
\end{equation}
\end{corollary}
\begin{proof}
By replacing $\varphi=90^{\circ}$ in (\ref{calctwistangle2}), we derive (\ref{calctwistangle7n}).
\end{proof}

\begin{corollary}
If $\varphi_{12}=\varphi_{34}=90^{\circ},$ then $\omega=\varphi.$
\end{corollary}
\begin{proof}
By replacing $\varphi_{12}=\varphi_{34}=90^{\circ}$ in (\ref{calctwistangle2}), we derive $\omega=\varphi.$
\end{proof}

\end{document}